\documentclass[12pt,reqno]{amsart}

\usepackage{amsthm, amsmath, amsfonts, amssymb, graphicx, tikz-cd, makecell, calrsfs, hyperref, tabularx}
\usepackage[shortlabels]{enumitem}
\usepackage{ltablex, comment}
\usepackage[capitalize,noabbrev]{cleveref} 

\newtheorem{theorem}{Theorem}[section]
\newtheorem{lemma}[theorem]{Lemma}

\newtheorem{corollary}[theorem]{Corollary}

\theoremstyle{definition}
\newtheorem{definition}[theorem]{Definition}
\newtheorem{example}[theorem]{Example}

\newtheorem{remark}[theorem]{Remark}

\newcounter{mycount}

\newcommand{\myref}[1]{\hyperref[#1]{#1}}

\def\Esa{{\sf{ES}}}

\def\Pries{{\sf{PS}}}

\def\GPries{{\sf{GPS}}}
\def\PGPries{{\sf{PGPS}}}
\def\PGPriesP{{\sf{PGPS}_P}}

\def\PGEsa{{\sf{PGES}}}
\def\PGEsaP{{\sf{PGES}_P}}
\def\PEsa{{\sf{PES}}}
\def\PEsaP{{\sf{PES}_P}}

\def\DLat{{\sf{DL}}}
\def\HA{{\sf{HA}}}

\def\DMSLat{{\sf{DMS}}}
\def\DMSLatP{{\sf{DMS}_P}}

\def\AlgFrm{{\sf{AlgFrm}}}

\def\CohFrm{{\sf{CohFrm}}}
\def\BrwMS{{\sf{BrwMS}}}
\def\BrwMSP{{\sf{BrwMS}_P}}
\def\BrwA{{\sf{BrwA}}}
\def\BrwAP{{\sf{BrwA}_P}}
\def\BrwFrm{{\sf{BrwFrm}}}
\def\BrwFrmJ{{\sf{BrwFrm}_J}}
\def\BrwArFrm{{\sf{BrwArFrm}}}
\def\BrwArFrmJ{{\sf{BrwArFrm}_J}}
\def\HeytFrm{{\sf{HeytFrm}}}

\def\Spec{{\sf{Spec}}}

\def\up{{\uparrow}}
\def\down{{\downarrow}}

\def\A{\mathcal{A}}
\def\I{\mathcal{I}}
\def\X{\mathcal{X}}
\def\V{\mathcal{V}^a}
\def\K{\mathcal{K}}
\def\Y{\mathcal{Y}}
\def\F{\mathcal{F}}
\def\Opt{{\sf{Opt}}}

\def\pf{{\sf{pf}}}
\def\pt{{\sf{pt}}}
\def\Idl{{\mathcal{J}}}

\def\OpUp{{\sf{OpUp}}}
\def\ClopUp{{\sf{ClopUp}}}

\newcommand{\cl}{{\sf{cl}}}
\newcommand{\Int}{{\sf{int}}}

\newcommand\nr[1]{\mathord{\not\mathrel{#1}}}

\begin{document}

\title{Heyting frames and Esakia duality}

\author{G.~Bezhanishvili}
\address{New Mexico State University}
\email{guram@nmsu.edu}

\author{L.~Carai}
\address{Universitat de Barcelona}
\email{luca.carai.uni@gmail.com}

\author{P.~J.~Morandi}
\address{New Mexico State University}
\email{pmorandi@nmsu.edu}


\begin{abstract}
We introduce the category of Heyting frames and show that it is equivalent to the category of Heyting algebras and dually equivalent to the category of Esakia spaces. This provides a frame-theoretic perspective on Esakia duality for Heyting algebras. We also generalize these results to the setting of Brouwerian algebras and Brouwerian semilattices by introducing the corresponding categories of Brouwerian frames and extending the above equivalences and dual equivalences. This provides a frame-theoretic perspective on generalized Esakia duality for Brouwerian algebras and Brouwerian semilattices.
\end{abstract}

\maketitle

\section{Introduction} \label{sec: intro}

In the early 1970s, two important duality theorems were established, by Priestley \cite{Pri70,Pri72} for bounded distributive lattices and by Esakia \cite{Esa74} for Heyting algebras. In both cases, the dual structures were special compact ordered spaces, which became known as Priestley spaces and Esakia spaces, respectively. As was shown by Cornish \cite{Cor75}, Priestley duality is closely related to Stone duality for bounded distributive lattices \cite{Sto37c}. Since Stone duals are spectral spaces, this opens the door for a frame-theoretic approach to Priestley duality. Indeed, the category $\DLat$ of bounded distributive lattices is equivalent to the category $\CohFrm$ of coherent frames, which are exactly the frames of open subsets of spectral spaces \cite[p.~65]{Joh82}. By \cite{Cor75}, the category $\Spec$ of spectral spaces is isomorphic to the category $\Pries$ of Priestley spaces. We thus obtain the following diagram connecting distributive lattices, coherent frames, and Priestley spaces. The arrow between $\DLat$ and $\CohFrm$ is an equivalence, the other two are dual equivalences, and the diagram commutes up to natural isomorphism (see \cref{sec: coherent frames} for details).   
\[
\begin{tikzcd}[column sep = 5pc]
\DLat \arrow[rr, leftrightarrow] \arrow[dr, leftrightarrow] && \CohFrm \arrow[dl, leftrightarrow] \\
& \Pries 
\end{tikzcd}
\]

Since Esakia duality is a restricted version of Priestley duality, it is natural to provide a frame-theoretic approach to it that is similar to the above approach to Priestley duality. For this purpose, we introduce the notion of a Heyting frame. To justify the definition, we recall that a frame $L$ is coherent if the set $K(L)$ of compact elements of $L$ is a bounded sublattice of $L$. Since every frame is a Heyting algebra, we define $L$ to be a Heyting frame if $K(L)$ is a Heyting subalgebra of $L$. Clearly a Heyting frame is a coherent frame. We show that the above correspondence between bounded distributive lattices, coherent frames, and Priestley spaces restricts to the same correspondence between Heyting algebras, Heyting frames, and Esakia spaces. 

There are various natural morphisms $\alpha : L\to M$ to consider between two Heyting frames. If we consider coherent frame homomorphisms (those frame homomorphisms that send compact elements to compact elements), the resulting category is a full subcategory of $\CohFrm$. But such morphisms don't take into account Heyting implication. Thus, we also consider those coherent frame homomorphisms that preserve implication on $K(L)$ as well as those that preserve implication on the entire $L$. In addition, we consider complete lattice homomorphisms as well as complete Heyting homomorphisms. This results in various categories of Heyting frames. For each we describe the corresponding categories of Esakia spaces and Heyting algebras. This yields the desired frame-theoretic approach to Esakia duality, with various natural morphisms at play.

We also provide a frame-theoretic approach to generalized Esakia duality for Brouwerian algebras \cite{BMR17} and Brouwerian semilattices \cite{BJ13}. We recall that Brouwerian algebras are Heyting algebras except possibly without bottom and Brouwerian semilattices are Brouwerian algebras except possibly without join. The dual spaces of Brouwerian algebras are pointed Esakia spaces, and those of Brouwerian semilattices are pointed generalized Esakia spaces. We extend the notion of Heyting frames to that of Brouwerian frames and Brouwerian arithmetic frames. This requires to work with algebraic and arithmetic frames instead of coherent frames. For compact elements of an algebraic frame $L$ to form a Brouwerian semilattice, we need to work with the dual order on $K(L)$. We thus define an algebraic frame $L$ to be Brouwerian provided the dual $K(L)^d$ is a Brouwerian semilattice. Since implication on $K(L)^d$ becomes co-implication on $K(L)$, it no longer makes sense to talk about $K(L)^d$ being a Brouwerian sub-semilattice of $L$ since co-implication may not exist on $L$. 

We show that Brouwerian frames play the same role in generalized Esakia duality for Brouwerian semilattices as Heyting frames play in Esakia duality, and that a similar role is played by Brouwerian arithmetic frames for Brouwerian algebras and pointed Esakia spaces. We conclude the article by discussing how the results about Heyting frames are related to the corresponding results about Brouwerian frames.

\section{Preliminaries} \label{sec: coherent frames}

We start by briefly describing Priestley duality \cite{Pri70,Pri72}. Let $X$ be a poset. For $S\subseteq X$, we let
\[
\up S = \{ x \in X : s \le x \mbox{ for some } s \in S \} \mbox{ and } \down S = \{ x \in X : x \le s \mbox{ for some } s \in S \}.
\]
If $S=\{s\}$ is a singleton, then we simply write $\up s$ and $\down s$. We call $S$ an {\em upset} if $\up S=S$ and a {\em downset} if $\down S=S$. A {\em Priestley space} is a poset $X$ equipped with a compact topology such that the {\em Priestley separation axiom} holds: If $x\not\le y$, then there is a clopen upset $U$ such that $x\in U$ and $y\notin U$. A \emph{Priestley morphism} is a continuous order-preserving map. Let $\Pries$ be the category of Priestley spaces and Priestley morphisms.  Let also $\DLat$ be the category of bounded distributive lattices and bounded lattice homomorphisms. 

\begin{theorem}[Priestley duality] \label{thm: Priestley}
$\DLat$ is dually equivalent to $\Pries$.
\end{theorem}

The contravariant functors establishing Priestley duality are constructed as follows. For a Priestley space $X$, let $\ClopUp(X)$ be the bounded distributive lattice of clopen upsets of $X$. The contravariant functor $\ClopUp:\Pries\to\DLat$ sends $X \in \Pries$ to $\ClopUp(X)$ and a $\Pries$-morphism $f :X \to Y$ to the bounded lattice homomorphism $f^{-1} : \ClopUp(Y) \to \ClopUp(X)$. For $A \in \DLat$ let $X_A$ be the set of prime filters of $A$ ordered by inclusion and equipped with the topology whose basis is $\{ \varphi(a)\setminus\varphi(b) : a,b \in A \}$, where $\varphi : A \to \wp(X_A)$ is the {\em Stone map} $\varphi(a) = \{ x \in X_A : a \in x\}$. Then $X_A$ is a Priestley space and the contravariant functor $\pf:\DLat \to \Pries$ sends $A\in\DLat$ to $X_A$ and a $\DLat$-morphism $h : A \to B$ to $h^{-1} : X_B \to X_A$.

Esakia duality \cite{Esa74} is a restricted version of Priestley duality. We recall that an {\em Esakia space} is a Priestley space $X$ in which $\down U$ is clopen for each clopen $U$. An \emph{Esakia morphism} is a continuous map $f:X\to Y$ satisfying $\down f^{-1}(y)=f^{-1}(\down y)$ for each $y\in Y$.\footnote{Maps $f:X\to Y$ between posets satisfying the latter condition are often called {\em p-morphisms} or {\em bounded morphisms}.} Let $\Esa$ be the category of Esakia spaces and Esakia morphisms.  We also recall that a \emph{Heyting algebra} is a bounded distributive lattice $A$ such that $\wedge$ has a residual $\to$ satisfying $a \wedge c \le b$ iff $c \le a \to b$. The operation $\to$ is often referred to as {\em implication}. Let $\HA$ be the category of Heyting algebras and Heyting algebra homomorphisms.
 
\begin{theorem}[Esakia duality] \label{thm: Esakia}
$\HA$ is dually equivalent to $\Esa$.
\end{theorem}

Note that $\HA$ is a non-full subcategory of $\DLat$, $\Esa$ is a non-full subcategory of $\Pries$, and Esakia duality is established by restricting the contravariant functors $\pf$ and $\ClopUp$ establishing Priestley duality. 

The following lemma is well known in Esakia duality and will be used subsequently (see, e.g., \cite[Thms.~4.2 and 4.3]{DG03}).

\begin{lemma} \label{lem: Esakia duality lemma}
\hfill
\begin{enumerate}[label=$(\arabic*)$, ref=\thelemma(\arabic*)]
\item \label[lemma]{lem: Esakia duality lemma(1)} 
A Priestley space $X$ is an Esakia space iff $\ClopUp(X)$ is a Heyting algebra, in which case the implication of $U,V\in\ClopUp(X)$ is calculated by $U\to V = X \setminus \down (U\setminus V)$.  
\item \label[lemma]{lem: Esakia duality lemma(2)}
Let $X, Y$ be Esakia spaces and $f : X \to Y$ a $\Pries$-morphism. The following conditions are equivalent:
\begin{enumerate}[label = $(\alph*)$] 
\item $f$ is an $\Esa$-morphism. 
\item $f^{-1}(\down U) = \down f^{-1}(U)$ for each clopen subset $U$ of $Y$.
\item $f^{-1}$ is a Heyting homomorphism.
\end{enumerate}
\end{enumerate}
\end{lemma}

We recall (see, e.g., \cite[p.~10]{PP12}) that a {\em frame} is a complete lattice $L$ satisfying the infinite distributive law $a\wedge\bigvee S= \bigvee \{a\wedge s : s \in S \}$ for each $S\subseteq L$. An element $a\in L$ is {\em compact} if $a \le \bigvee S$ implies $a\le\bigvee T$ for some finite $T\subseteq S$. Let $K(L)$ be the set of compact elements of $L$. Then $K(L)$ is a join-subsemilattice of $L$. We say that $L$ is {\em algebraic} if $K(L)$ is join-dense in $L$ (that is, each element of $L$ is a join of compact elements), and that $L$ is {\em coherent} if $L$ is algebraic and in addition $K(L)$ is a bounded sublattice of $L$.

A map $\alpha : L\to M$ between frames is a {\em frame homomorphism} if $\alpha$ preserves finite meets and arbitrary joins. Let $\AlgFrm$ be the category of algebraic frames and frame homomorphisms that preserve compact elements (that is, $a\in K(L)$ implies $\alpha(a)\in K(M)$). Let also $\CohFrm$ be the full subcategory of $\AlgFrm$ consisting of coherent frames. 

\begin{theorem} \cite[p.~65]{Joh82}
$\DLat$ is equivalent to $\CohFrm$.
\end{theorem}

The functors establishing this equivalence are constructed as follows. The functor $\Idl:\DLat\to\CohFrm$ 
sends $A\in\DLat$ to the frame $\Idl(A)$ of ideals of $A$, and a $\DLat$-morphism $h:A \to B$ to the $\CohFrm$-morphism $h^*:\Idl(A) \to \Idl(B)$ given by $h^*(I)=\down h[I]$. The functor $\K : \CohFrm\to\DLat$ sends $L\in\CohFrm$ to its bounded sublattice $K(L)$ and a $\CohFrm$-morphism $\alpha : L\to M$ to its restriction to $K(L)$.

Since frames are precisely complete Heyting algebras (see, e.g., \cite[p.~12]{Esa19}), 
the ideal frame of $A\in\DLat$ is a Heyting algebra. We will use the following description of implication on $\Idl(A)$. 
\begin{lemma} \label{lem: implication in ideals}
Let $A \in \DLat$. Then implication on $\Idl(A)$ is calculated by the following formula
\[
I \to J = \{ a \in A : a \wedge b \in J \ \forall b \in I\} = \{ a \in A : \down a \cap I \subseteq J \}.
\]
\end{lemma}

\begin{proof}
Since the second equality is straightforward,  we only verify the first equality. Set $E = \{ a \in A : a \wedge b \in J \ \forall b \in I\}$. Since $A\in\DLat$, it is straightforward to see that $E\in\Idl(A)$. We show that $E=I\to J$. For this we must show that $E$ is the largest ideal of $A$ satisfying $I\cap E\subseteq J$. To see that $I\cap E\subseteq J$, let $a \in I \cap E$. Then $a\in I$ and $a\in E$. The latter implies that $a\wedge b \in J$ for all $b\in I$. In particular, since $a\in I$, we have $a=a\wedge a \in J$. Therefore, $I \cap E \subseteq J$. To see that $E$ is the largest such, let $N\in\Idl(A)$ with $I \cap N \subseteq J$. Let $a \in N$ and $b \in I$ be arbitrary. Then $a \wedge b \in I \cap N \subseteq J$, so $a \in E$. Thus, $N\subseteq E$, and hence $I \to J = E$. 
\end{proof}

Since $\DLat$ is equivalent to $\CohFrm$ and dually equivalent to $\Pries$, we obtain that $\CohFrm$ is dually equivalent to $\Pries$. The contravariant functors establishing this dual equivalence can be constructed as follows. Let $X\in\Pries$. Since the frame of ideals of $\ClopUp(X)$ is isomorphic to the frame $\OpUp(X)$ of open upset of $X$ (see, e.g., \cite[p.~54]{Pri84} or \cite[p.~385]{BBGK10}), the contravariant functor $\OpUp:\Pries\to\CohFrm$ sends $X\in\Pries$ to $\OpUp(X)\in\CohFrm$ and a $\Pries$-morphism $f:X\to Y$ to the $\CohFrm$-morphism $f^{-1}:\OpUp(Y)\to\OpUp(X)$. 

To describe the contravariant functor $\pt : \CohFrm \to \Pries$, we recall (see, e.g., \cite[p.~13]{PP12}) that a {\em point} of a frame $L$ is a completely prime filter $P$ (that is, $P$ is a filter such that $\bigvee S\in P$ implies $S\cap P\ne\varnothing$). Let $\pt(L)$ be the set of points of $L$. For each $a \in L$ let $\zeta(a) = \{ P \in \pt(L) : a \in P\}$. It is well known (see, e.g., \cite[p.~15]{PP12}) that $\tau = \{ \zeta(a) : a \in L\}$ is a topology on $\pt(L)$. Let $L\in\CohFrm$. Then $(\pt(L), \tau)$ is a spectral space (see, e.g., \cite[p.~65]{Joh82}), and $(\pt(L), \pi, \subseteq)$ is the corresponding Priestley space, where $\pi$ is the patch topology of $\tau$ (see, e.g., \cite{Cor75}). The contravariant functor $\pt:\CohFrm\to\Pries$ sends $L\in\CohFrm$ to the Priestley space $(\pt(L), \pi, \subseteq)$, and a $\CohFrm$-morphism $\alpha : L\to M$ to $\alpha^{-1}:\pt(M)\to\pt(L)$. 

For $L \in \CohFrm$ the Priestley spaces $(\pt(L), \pi, \subseteq)$ and $\pf(K(L))$ of prime filters of $K(L)$ are isomorphic in $\Pries$. This isomorphism is obtained by sending each point $P$ of $L$ to $P\cap K(L)$ and each prime filter $F$ of $K(L)$ to its upset $\up F$ in $L$.
We thus obtain the following well-known result, which provides a frame-theoretic perspective on Priestley duality: 

\begin{theorem} \label{thm: Pries and Coh}
There is an equivalence of categories between $\DLat$ and $\CohFrm$, a dual equivalence between $\CohFrm$ and $\Pries$, and the following diagram commutes up to natural isomorphism.
\[
\begin{tikzcd}[column sep = 5pc]
\DLat \arrow[rr, shift left = .5ex, "\Idl"] \arrow[dr, "{\sf pf}"] & & \CohFrm \arrow[ll, shift left = .5ex, "{\K}"] \arrow[dl, "{\sf pt}"'] \\
& \Pries \arrow[ul, shift left = 1.0ex, "\ClopUp"] \arrow[ur, shift right = 1.0ex, "\OpUp"'] &
\end{tikzcd}
\]
\end{theorem}

Since Esakia duality is a restricted version of Priestley duality, it is natural to restrict \cref{thm: Pries and Coh} to the category $\HA$ of Heyting algebras. Then $\Pries$ restricts to $\Esa$. In the next section we will introduce the category $\HeytFrm$ of Heyting frames and show that it plays the same role for $\HA$ and $\Esa$ that $\CohFrm$ plays for $\DLat$ and $\Pries$. 

\section{Heyting frames} \label{sec: Heyting frames}

\begin{definition} \label{def: HeytFrm}
We call a frame $L$ a \emph{Heyting frame} if $K(L)$ is a Heyting subalgebra of $L$.
\end{definition}

Since $K(L)$ being a Heyting subalgebra implies that $K(L)$ is a bounded sublattice of $L$, we see that every Heyting frame is a coherent frame.

\begin{theorem} \label{thm: HeytFrm}
Let $L$ be a coherent frame. Then $L$ is a Heyting frame iff the bounded sublattice $K(L)$ of $L$ is a Heyting algebra.
\end{theorem}

\begin{proof}
The left-to-right implication is obvious. For the right-to-left implication, let $\to$ be the implication on $L$. Since the bounded sublattice $K(L)$ is a Heyting algebra, it has an implication $\to'$. It is sufficient to show that $a \to b = a \to' b$ for each $a, b \in K(L)$. Because $a \wedge (a \to' b) \le b$, we have $a \to' b \le a \to b$. To see the reverse inequality, since $L$ is coherent, $a \to b = \bigvee \{ k \in K(L) : k \le a \to b\}$. Let $k \in K(L)$ with $k \le a \to b$. Because $a \wedge (a \to b) \le b$, we have $a \wedge k \le b$, so $k \le a \to' b$. Thus, $a \to b \le a \to' b$, and hence equality holds. Consequently, $L$ is a Heyting frame. 
\end{proof}

Since each $A \in \DLat$ is isomorphic to $K(\Idl(A))$, as an immediate consequence of \cref{thm: HeytFrm} we obtain:

\begin{corollary}
Let $A \in \DLat$. Then $A$ is a Heyting algebra iff $\Idl(A)$ is a Heyting frame.
\end{corollary}

We next give a dual characterization of Heyting frames.
 
\begin{theorem} \label{lem: L Heyting iff X Esakia}
\hfill
\begin{enumerate}[label = $(\arabic*)$]
\item If $X \in \Pries$, then $X$ is an Esakia space iff $\OpUp(X)$ is a Heyting frame.
\item If $L \in \CohFrm$, then $L$ is a Heyting frame iff $\pt(L)$ is an Esakia space.
\end{enumerate}
\end{theorem}

\begin{proof}
(1) Let $X\in\Pries$. By \cref{lem: Esakia duality lemma(1)}, $X$ is an Esakia space iff $\ClopUp(X)$ is a Heyting algebra. Since $\ClopUp(X)$ is the set of compact elements of the coherent frame $\OpUp(X)$, it follows from \cref{thm: HeytFrm} that $X$ is an Esakia space iff $\OpUp(X)$ is a Heyting frame.

(2) Let $L\in\CohFrm$. Dualizing \cref{lem: Esakia duality lemma(1)}, we have that $A\in\DLat$ is a Heyting algebra iff $\pf(A)$ is an Esakia space. 
Thus, $K(L)$ is a Heyting algebra iff $\pf(K(L))$ is an Esakia space. 
Since $\pt(L)$ is order-homeomorphic to $\pf(K(L))$ (see the paragraph before \cref{thm: Pries and Coh}), 
we conclude by \cref{thm: HeytFrm} that $L$ is a Heyting frame iff 
$\pt(L)$ is an Esakia space.
\end{proof}

We thus obtain a one-to-one correspondence between Heyting algebras, Heyting frames, and Esakia spaces. We next consider various morphisms between Heyting frames.

\begin{definition}
\hfill
\begin{enumerate}
\item Let $\HeytFrm^-$ be the category of Heyting frames and coherent frame homomorphisms.
\item Let $\HeytFrm$ be the category of Heyting frames and coherent frame homomorphisms which restrict to $\HA$-morphisms on compact elements.
\item Let $\HeytFrm^+$ be the category of Heyting frames and coherent frame homomorphisms preserving implication.
\item Let $\HeytFrm^*$ be the category of Heyting frames and complete lattice homomorphisms preserving compact elements.
\item Let $\HeytFrm^\dagger$ be the category of Heyting frames and complete Heyting homomorphisms preserving compact elements.
\end{enumerate}
\end{definition}

We clearly have the following inclusions of categories. Since all five categories have same objects, each inclusion is a wide inclusion. That each inclusion in the diagram is proper, and that $\HeytFrm^*$ is incomparable with $\HeytFrm$ and $\HeytFrm^+$ will follow from a similar result for the corresponding categories of Esakia spaces (see \cref{example}).

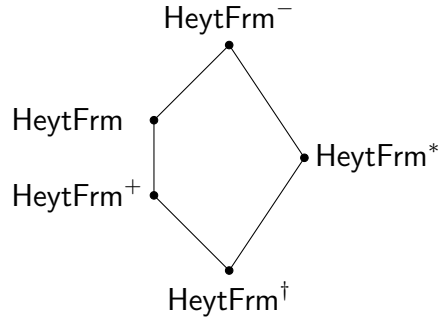
\begin{center} \label{diagram: Heyting Frames}
\begin{tikzpicture}
\node [below] at (0,0) {$\HeytFrm^\dagger$};
\draw [fill] (0,0) circle[radius = .05];
\node [left] at (-1, 1) {$\HeytFrm^+$};
\draw [fill] (-1, 1) circle[radius = .05];
\node [left] at (-1, 2) {$\HeytFrm^{\phantom{+}}$};
\draw [fill] (-1, 2) circle[radius = .05];
\node [right] at (1, 1.5) {$\HeytFrm^*$};
\draw [fill] (1, 1.5) circle[radius = .05];
\node [above] at (0,3) {$\HeytFrm^-$};
\draw [fill] (0,3) circle[radius = .05];
\draw (0,0) -- (-1, 1);
\draw (0,0) -- (1, 1.5);
\draw (-1, 1) -- (-1, 2);
\draw (1, 1.5) -- (0,3);
\draw (-1, 2) -- (0,3);
\end{tikzpicture}
\end{center}

We next introduce the corresponding categories of Esakia spaces. We recall (see, e.g., \cite[p.~94]{Har77}) that a constructible subset of a topological space is a finite union of sets of the form $U \setminus V$ where $U, V$ are open sets. By analogy, we introduce the following notion for Esakia spaces.

\begin{definition}
Let $X$ be an Esakia space. We call a subset $E$ of $X$ an \emph{Esakia-constructible}, or simply {\em E-constructible}, subset of $X$ if $E$ is a finite union of sets of the form $U \setminus V$ where $U, V$ are open upsets. 
\end{definition}

\begin{definition}
\hfill
\begin{enumerate}
\item Let $\Esa^-$ be the category of Esakia spaces and Priestley morphisms.
\item Let $\Esa$ be the category of Esakia spaces and Esakia morphisms.
\item Let $\Esa^+$ be the category of Esakia spaces and Priestley morphisms $f$ satisfying $f^{-1}(\down \cl E) = \down \cl f^{-1}(E)$ for each E-constructible subset $E$.
\item Let $\Esa^*$ be the category of Esakia spaces and Priestley morphisms $f$ satisfying $f^{-1}(\down \cl D) = \down \cl f^{-1}(D)$ for each downset $D$.
\item Let $\Esa^\dagger$ be the category of Esakia spaces and maps that are both $\Esa^+$ and $\Esa^*$-morphisms.
\end{enumerate}
\end{definition}

\begin{remark} \label{rem: conv}
Since $\down, \cl$, and $f^{-1}$ preserve finite unions, in the definition of $\Esa^+$ we may assume that $E = U \setminus V$ with $U, V$ open upsets. 
\end{remark}

We thus obtain the following wide inclusions of categories. The only inclusion that is not obvious is that $\Esa^+$ is a wide subcategory of $\Esa$. We prove this in \cref{lem: Esa+ is a sub of Esa}. 

\begin{center}
\begin{tikzpicture}
\node [below] at (0+1.75,0) {$\Esa^\dagger$};
\draw [fill] (0+1.75,0) circle[radius = .05];
\node [left] at (-1+1.75, 1) {$\Esa^+$};
\draw [fill] (-1+1.75, 1) circle[radius = .05];
\node [left] at (-1+1.75, 2) {$\Esa^{\phantom{+}}$};
\draw [fill] (-1+1.75, 2) circle[radius = .05];
\node [right] at (1+1.75, 1.5) {$\Esa^*$};
\draw [fill] (1+1.75, 1.5) circle[radius = .05];
\node [above] at (0+1.75,3) {$\Esa^-$};
\draw [fill] (0+1.75,3) circle[radius = .05];
\draw (0+1.75,0) -- (-1+1.75, 1);
\draw (0+1.75,0) -- (1+1.75, 1.5);
\draw (-1+1.75, 1) -- (-1+1.75, 2);
\draw (-1+1.75, 2) -- (0+1.75,3);
\draw (1+1.75, 1.5) -- (0+1.75,3);
\end{tikzpicture}
\end{center}

\begin{lemma} \label{lem: Esa+ is a sub of Esa}
$\Esa^+$ is a wide subcategory of $\Esa$.
\end{lemma}

\begin{proof}
Let $f : X \to Y$ be an $\Esa^+$-morphism. To see that $f$ is an $\Esa$-morphism, 
by \cref{lem: Esakia duality lemma(2)} it suffices to show that $f^{-1}(\down U) = \down f^{-1}(U)$ for each clopen subset $U$ of $Y$. 
Since $\{ C\setminus D : C,D\in\ClopUp(Y) \}$ is a basis for the topology on $Y$ and $U$ is clopen, it is a finite union of $U_i \setminus V_i$ with $U_i, V_i \in \ClopUp(Y)$. 
Therefore, $U$ is E-constructible, and hence
\[
f^{-1}(\down U) = f^{-1}(\down \cl U) = \down \cl f^{-1}(U) = \down f^{-1}(U)
\]
since $f$ is an $\Esa^+$-morphism.
\end{proof}

In the next example we show that each inclusion in the diagram above is proper, and that $\Esa^*$ is incomparable with $\Esa$ and $\Esa^+$. The duality results we obtain later in the section will then imply that the same holds for the corresponding categories of Heyting frames.

\begin{example}
\label{example}
\hfill

\begin{enumerate}[leftmargin = 0.3in]
\item\label{Esa- not Esa} 
Let $X_1 = \{x, y\}$ be the two-element chain where $x < y$. We view $X_1$ as an Esakia space with the discrete topology. Let $f : X_1 \to X_1$ be defined by $f(x) = f(y) = x$. Then $f$ is order-preserving. Therefore, $f^{-1}(D)$ is a downset for each downset $D$ of $X_1$. Since the topology on $X_1$ is discrete, this implies that $f$ is an $\Esa^*$-morphism.
 However, $f$ is not an $\Esa$-morphism because $f^{-1}(\down y) = X_1$ but $\down f^{-1}(y) = \varnothing$. Therefore, $\Esa^*$ is not a subcategory of $\Esa$. Since $\Esa^+$ is a wide subcategory of $\Esa$ and $\Esa^*$ is a wide subcategory of $\Esa^-$, it follows that $\Esa^*$ is not a subcategory of $\Esa^+$ and $\Esa$ is a proper wide subcategory of $\Esa^-$.

\item Let $E$ and $O$ be the sets of even and odd natural numbers. We set $X_2 = \mathbb{N} \cup \{\infty_E, \infty_O \}$ to be the two-point compactification of the discrete space $\mathbb{N}$, where $\infty_E$ is the limit point of $E$ and $\infty_O$ the limit point of $O$. We also set $X_3 =\mathbb{N} \cup \{\infty\}$ to be the one-point compactification of $\mathbb{N}$. Clearly $X_2$ and $X_3$ are Stone spaces, and hence they are Esakia spaces with the trivial order. Define $f : X_2 \to X_3$ by $f(n) = n$ for each $n \in \mathbb{N}$ and $f(\infty_E) = f(\infty_O) = \infty$. Since $f$ is continuous and the orders are trivial, $f$ is obviously an $\Esa$-morphism. On the other hand, 
\[
f^{-1}(\down \cl E) = f^{-1}(\cl E) = f^{-1}(E \cup \{\infty\}) = E \cup \{\infty_E, \infty_O \}
\]
while $\down \cl f^{-1}(E) = \cl(E) = E \cup \{\infty_E\}$. Therefore, since $E$ is E-constructible, $f$ is not an $\Esa^+$-morphism. Thus, $\Esa^+$ is a proper wide subcategory of $\Esa$.

\item \label[example]{Esa+ not Esa*} 
Let $X_4 = \mathbb{N} \cup \{\infty\}$ be the one-point compactification of the discrete space $\mathbb{N}$. Define $\le$ on $X_4$ by $u \le v$ iff $u = v$ or $v = \infty$. It is easy to verify that $X_4$ is an Esakia space. Recalling $X_1$ from (\ref{Esa- not Esa}), define $f : X_1 \to X_4$ by $f(x) = 0$ and $f(y) = \infty$. It is elementary to see that $f$ is an $\Esa$-morphism. Clearly $\mathbb{N}$ is a downset of $X_4$ and $\down \cl \mathbb{N} = X_4$. Therefore, $f^{-1}(\down \cl \mathbb{N}) = X_1$. But $f^{-1}(\mathbb{N}) = \{x\}$, and so $\down \cl f^{-1}(\mathbb{N}) = \{x\}$. Thus, $f$ is not an $\Esa^*$-morphism. Consequently, $\Esa$ and $\Esa^*$ are incomparable. Since 
$\Esa$ is a wide subcategory of $\Esa^-$, it follows that 
$\Esa^*$ is a proper wide subcategory of $\Esa^-$. 

We next show that $f$ is an $\Esa^+$-morphism. Let $W=U\setminus V$ where $U,V$ are open upsets of $X_2$. If $\infty \in W$, then $\down \cl W = X_2$, so $f^{-1}(\down \cl W) = X_1$. Also, $y \in f^{-1}(W)$, so $\down \cl f^{-1}(W) = X_1$. If $\infty \notin W$, then $W$ must be a finite downset. Therefore, $f^{-1}(\down \cl W)$ is $\{x\}$ or $\varnothing$ depending on whether $0 \in W$. 
Similarly, $\down \cl f^{-1}(W)$ is $\{x\}$ or $\varnothing$,
so $f$ is an $\Esa^+$-morphism by \cref{rem: conv}. Thus, $\Esa^+$ and $\Esa^*$ are also incomparable. 

\item From (\ref{Esa- not Esa}) and (\ref{Esa+ not Esa*}) it follows that $\Esa^\dagger$ is a proper wide subcategory of both $\Esa^*$ and $\Esa^+$.
\end{enumerate}
\end{example}

We next introduce the corresponding categories of Heyting algebras. Let $\HA^-$ be the full subcategory of $\DLat$ consisting of Heyting algebras. The following is an immediate consequence of the equivalence of $\DLat$ and $\CohFrm$, the dual equivalence of $\CohFrm$ and $\Pries$, and \cref{lem: L Heyting iff X Esakia}.

\begin{theorem} \label{thm: - case}
$\HA^-$ and $\HeytFrm^-$ are equivalent and dually equivalent to $\Esa^-$.
\end{theorem}

We next establish that $\HeytFrm$ is equivalent to $\HA$ and hence dually equivalent to $\Esa$. For this we require the following lemma. We recall that for $A \in \DLat$, compact elements of $\Idl(A)$ are precisely the principal ideals $\down a$ for $a \in A$. We also recall that for a $\DLat$-morphism $h : A \to B$, the frame homomorphism $h^*:\Idl(A) \to \Idl(B)$ is given by $h^*(I)=\down h[I]$.

\begin{lemma} \label{lem: when is f an Esakia morphism}
Let $h:A\to B$ be a $\DLat$-morphism between Heyting algebras. Then $h$ is a Heyting homomorphism iff $h^*$ preserves implication of compact elements.
\end{lemma}

\begin{proof}
Let $a, b \in A$. Since $A$ is a Heyting algebra, by \cref{lem: implication in ideals},
\begin{align*}
\down a \to \down b &= 
\{ x \in A : x \wedge y \le b \ \forall y \in \down a\} \\
&= \{ x \in A : x \wedge a \le b\}= \down (a \to b).
\end{align*}
Since $h$ is order-preserving, $h[\down x] \subseteq \down h(x)$, and hence $\down h[\down x] = \down h(x)$ for each $x \in A$. Therefore, using $\down a \to \down b = \down (a \to b)$, we have
\[
h^*(\down a \to \down b) = h^*\down (a \to b)= \down h[\down(a \to b)] = \down h(a \to b).
\]
On the other hand, a similar calculation gives
\[
h^*(\down a) \to h^*(\down b) = \down h[\down a] \to \down h[\down b] = \down h(a) \to \down h(b) = \down (h(a) \to h(b)).
\]
Therefore,
\begin{align*}
h^*(\down a \to \down b) = h^*(\down a) \to h^*(\down b) &\textrm{\ \ iff\ \  } \down h(a \to b) = \down (h(a) \to h(b)) \\
&\textrm{\ \ iff\ \ } h(a \to b) = h(a) \to h(b). 
\end{align*}
Thus, $h$ is a Heyting homomorphism iff $h^*$ preserves implication of compact elements.
\end{proof} 

It follows from 
\cref{lem: Esakia duality lemma(2)} that if $f$ is a Priestley morphism between Esakia spaces, then $f$ is an Esakia morphism iff $f^{-1}$ preserves implication on clopen upsets. 
We use this fact together with \cref{lem: when is f an Esakia morphism} to show that the equivalence and dual equivalence of \cref{thm: - case} restrict to yield \cref{thm: HA case}.
For this we note the following:

\begin{remark} \label{rem: category theory}
If $F : \mathcal{C} \to \mathcal{D}$ is part of an equivalence or dual equivalence and $\mathcal{C}'$ (resp.~$\mathcal{D}'$) is a subcategory of $\mathcal{C}$ (resp.~$\mathcal{D}$), then to see that $F$ restricts to part of an equivalence or dual equivalence between $\mathcal{C}'$ and $\mathcal{D}'$, it suffices to verify the following three conditions.
\begin{enumerate}
\item \label[remark]{rem1}If $A \in \mathcal{C}$, then $A \in \mathcal{C}'$ iff $F(A) \in \mathcal{D}'$.
\item \label[remark]{rem2}If $f$ is a morphism in $\mathcal{C}$, then $f$ is a morphism in $\mathcal{C}'$ iff $F(f)$ is a morphism in $\mathcal{D}'$.
\item \label[remark]{rem3}Isomorphisms in $\mathcal{C}$ between objects in $\mathcal{C}'$ are isomorphisms in $\mathcal{C}'$, and the same for $\mathcal{D}$ and $\mathcal{D}'$.
\end{enumerate}
\end{remark}

This remark will also be used to prove \cref{thm: + case,thm: * case} and \cref{cor: + case,cor: * case,thm: dagger case}.

\begin{theorem} \label{thm: HA case}
$\HeytFrm$ is equivalent to $\HA$ and dually equivalent to $\Esa$.
\end{theorem}

\begin{proof}
We first show that $\HeytFrm$ is equivalent to $\HA$.  It is obvious that Condition~(\ref{rem1}) of \cref{rem: category theory} is satisfied since $\HeytFrm$ is a wide subcategory of $\HeytFrm^-$ and $\HA$ is a wide subcategory of $\HA^-$.
\cref{lem: when is f an Esakia morphism} shows that a $\DLat$-morphism $h$ between Heyting algebras is a $\HA$-morphism iff $h^*$ is a $\HeytFrm$-morphism. Therefore, Condition~(\ref{rem2}) is satisfied. That Condition~(\ref{rem3})  is also satisfied follows from the fact that isomorphisms in all categories of frames and Heyting algebras that we consider are poset isomorphisms. Thus, the equivalence of \cref{thm: - case} restricts to an equivalence between $\HA$ and $\HeytFrm$. 

Next we show that $\Esa$ is dually equivalent to $\HeytFrm$. Again, Condition~(\ref{rem1}) is obvious.
\cref{lem: Esakia duality lemma}(3) implies that a $\Pries$-morphism $f$ between Esakia spaces is an $\Esa$-morphism iff $f^{-1}:\ClopUp(Y)\to\ClopUp(X)$ is a $\HA$-morphism, which is equivalent to $f^{-1}:\OpUp(Y)\to\OpUp(X)$ being a $\HeytFrm$-morphism. Therefore, Condition~(\ref{rem2}) is satisfied. That Condition~(\ref{rem3}) is also satisfied follows from the fact that isomorphisms in all categories of Priestley and Esakia spaces that we consider are order-homeomorphisms (that is, poset isomorphisms and homeomorphisms).
Thus, the duality of \cref{thm: - case} between $\Esa^-$ and $\HeytFrm^-$ restricts to a duality between $\Esa$ and $\HeytFrm$.
\end{proof}

It is natural to work not only with those frame homomorphisms that preserve implication on compact elements, but also with those that are $\HA$-morphisms. This results in a restricted version of Esakia duality, which we describe next. For this we recall (see, e.g., \cite[p.~16]{Esa19}) that if $X$ is a topological space, then implication in the frame of open subsets of $X$ is given by
\[
U \to V = X \setminus \cl(U\setminus V). 
\]
If $X$ is a Priestley space, then $\OpUp(X)$ is exactly the frame of open sets of the corresponding spectral topology. Since the closure in the spectral topology is $\down \cl$ (see, e.g., \cite[Lem.~6.5(1)]{BBGK10}), implication in $\OpUp(X)$ is given by
\begin{equation}
U \to V = X \setminus \down \cl(U \setminus V). \label{eqn: implication}
\end{equation}

\begin{lemma} \label{lem: when does 1/f preserve implication}
Let $f:X \to Y$ be a Priestley morphism between Esakia spaces. Then $f^{-1} : \OpUp(Y) \to \OpUp(X)$ is a $\HA$-morphism iff $f^{-1} \down \cl E = \down \cl f^{-1} E$ for each E-constructible subset $E$ of $Y$.
\end{lemma}

\begin{proof}
Let $U,V$ be open upsets of $Y$. By \cref{eqn: implication},
\[
f^{-1}(U \to V) = f^{-1}(Y \setminus \down \cl(U \setminus V)) = X \setminus f^{-1} \down \cl(U \setminus V)
\]
and
\[
f^{-1}(U) \to f^{-1}(V) = X \setminus \down \cl(f^{-1}(U) \setminus f^{-1}(V)) = X \setminus \down \cl f^{-1}(U \setminus V). 
\]
Therefore, $f^{-1}$ preserves implication iff $f^{-1} \down \cl E = \down \cl f^{-1}E$ for each $E = U \setminus V$ with $U, V \in \OpUp(Y)$. Since $f^{-1}$, $\down$, and $\sf{cl}$ preserve finite unions, $f^{-1}$ is a $\HA$-morphism 
iff $f^{-1} \down \cl E = \down \cl f^{-1}E$ for each E-constructible subset $E$ of $Y$.
\end{proof}

\begin{theorem} \label{thm: + case}
$\HeytFrm^+$ is dually equivalent to $\Esa^+$.
\end{theorem}

\begin{proof}
\cref{lem: when does 1/f preserve implication} shows that a Priestley morphism $f : X \to Y$ is an $\Esa^+$-morphism iff $f^{-1} : \OpUp(Y) \to \OpUp(X)$ is a $\HeytFrm^+$-morphism. Therefore, the duality between $\HeytFrm^-$ and $\Esa^-$ restricts to a duality between $\HeytFrm^+$ and $\Esa^+$.
\end{proof}

\begin{definition} \label{rem: +}
Let $\HA^+$ be the wide subcategory of $\HA^-$ whose morphisms $h : A \to B$ have the property that $h^* : \Idl(A) \to \Idl(B)$ is a $\HA$-morphism.
\end{definition}

It follows from \cref{lem: when is f an Esakia morphism} that $\HA^+$ is a wide subcategory of $\HA$. As an immediate consequence of \cref{thm: - case,thm: + case} we obtain:

\begin{corollary} \label{cor: + case}
$\HeytFrm^+$ is equivalent to $\HA^+$ and dually equivalent to $\Esa^+$.
\end{corollary}

It is well known that a continuous map $f:X\to Y$ between topological spaces is open iff 
$f^{-1}$ commutes with the closure operator for all subsets of $Y$ (see, e.g., \cite[pp.~99--100]{RS63}). 
Since the closure in the spectral topology of a Priestley space is $\down\,\cl$, the condition of \cref{lem: when does 1/f preserve implication} says that $f^{-1}$ commutes with the spectral closure operator on E-constructible subsets of $Y$. We give an example showing that such a map may not be open in the spectral topologies of $X$ and $Y$. 

\begin{example}
Let $X =\{*\}$ be a one-point space and $X_4$ the one-point compactification considered in \cref{Esa+ not Esa*}. 
Define $f : X \to X_4$ by $f(*) = \infty$. Then $f$ is an $\Esa$-morphism and is not an open map since $\infty$ is not an isolated point of the spectral topology on $X_4$. On the other hand, observe that $X$ is discrete and open upsets of $X_4$ are clopen. Therefore, each E-constructible subset $E$ of $X_4$ is clopen. Thus, by \cref{lem: Esakia duality lemma(2)},
\[
\down \cl f^{-1}(E) = \down f^{-1}(E) = f^{-1}( \down E) = f^{-1}(\down \cl E).
\]
Consequently, $f$ satisfies the condition of \cref{lem: when does 1/f preserve implication}. 
\end{example}

We next turn our attention to homomorphisms that preserve arbitrary meets.

\begin{lemma} \label{lem: when does 1/f preserve meet}
Let $f:X \to Y$ be a Priestley morphism between Priestley spaces. Then $f^{-1} : \OpUp(Y) \to \OpUp(X)$ preserves arbitrary meets iff $f^{-1} \down \cl D = \down \cl f^{-1} D$ for any downset $D$ of $Y$.
\end{lemma}

\begin{proof}
The frame homomorphism $f^{-1}$ preserves arbitrary meets iff for each family $\{U_\alpha\}$ of open upsets of $Y$, we have
\[
f^{-1}\left(\bigwedge U_\alpha \right) = \bigwedge f^{-1}(U_\alpha).
\]
We recall that the largest upset contained in a set $E\subseteq Y$ is $\Box\,E=Y\setminus\down(Y\setminus E)$, where we use the box notation as it is common in modal logic (see, e.g., \cite{CZ97}). Since in every Priestley space the downset of a closed set is closed (see, e.g., \cite[Prop.~2.6]{Pri84}), if $E$ is open, then $\Box E$ is open. Therefore, $\bigwedge U_\alpha = \Box\left(\Int\bigcap U_\alpha\right)$. 
Thus, 
\[
f^{-1}\left(\bigwedge U_\alpha\right) = f^{-1}\left(\Box\,\Int \bigcap U_\alpha\right)
\]
and
\[
\bigwedge f^{-1}(U_\alpha) = \Box\,\Int \bigcap f^{-1}(U_\alpha) = \Box\,\Int f^{-1}\left(\bigcap U_\alpha\right). 
\]
Since in every Priestley space the downset of a point is closed, upsets are precisely intersections of open upsets. 
Thus, $f^{-1}$ preserves meets iff $f^{-1}(\Box\,\Int U) = \Box\,\Int f^{-1}(U)$ for each upset $U$ of $Y$. This happens iff $f^{-1}(Y \setminus \down \cl (Y\setminus U)) = X \setminus \down \cl(X\setminus f^{-1}(U))$, which happens iff $f^{-1} \down \cl (Y\setminus U) = \down \cl (X\setminus f^{-1}(U)) = \down \cl f^{-1}(Y\setminus U)$. Since each downset $D$ of $Y$ is of the form $Y\setminus U$ for some upset $U$ of $Y$, the result follows.
\end{proof}

\begin{theorem} \label{thm: * case}
$\HeytFrm^*$ is dually equivalent to $\Esa^*$.
\end{theorem}

\begin{proof}
It follows from \cref{lem: when does 1/f preserve meet} that a Priestley morphism $f : X \to Y$ is an $\Esa^*$-morphism iff $f^{-1} : \OpUp(Y) \to \OpUp(X)$ is a $\HeytFrm^*$-morphism. Therefore, the duality between $\HeytFrm^-$ and $\Esa^-$ restricts to a duality between $\HeytFrm^*$ and $\Esa^*$.
\end{proof}

\begin{remark}
It follows from \cref{Esa+ not Esa*,thm: * case} that $\HeytFrm^+$ and $\HeytFrm^*$ are incomparable. The situation changes when we restrict to subfit Heyting frames, where we recall (see, e.g, \cite[p.~73]{PP12}) that a frame $L$ is \emph{subfit} if whenever $a, b \in L$ with $a \not\le b$, there is $c \in L$ with $a \vee c = 1$ and $b \vee c \ne 1$. Indeed, if $L,M$ are Heyting frames and $L$ is subfit, then it follows from \cite[Prop.~V.1.8]{PP12} that each $\HeytFrm^*$-morphism $\alpha : L\to M$ is a $\HeytFrm^+$-morphism. 
\end{remark}

\begin{definition} \label{rem: *}
Let $\HA^*$ be the wide subcategory of $\HA^-$ whose morphisms $h : A \to B$ in addition satisfy the property that $h^* : \Idl(A) \to \Idl(B)$ is a complete lattice homomorphism.  
\end{definition}

As an immediate consequence of \cref{thm: - case,thm: * case} we obtain:

\begin{corollary} \label{cor: * case} 
$\HeytFrm^*$ is equivalent to $\HA^*$ and dually equivalent to $\Esa^*$.
\end{corollary}

\begin{definition} \label{rem: dagger}
Let $\HA^\dagger=\HA^+\cap\HA^*$.
\end{definition}

Putting \cref{cor: + case,cor: * case} together yields:

\begin{corollary} \label{thm: dagger case}
$\HeytFrm^\dagger$ is equivalent to $\HA^\dagger$ and dually equivalent to $\Esa^\dagger$.
\end{corollary}

We thus obtain the following diagram of equivalences and dual equivalences involving the various categories of Heyting algebras, Heyting frames, and Esakia spaces considered in this section. Unlabeled arrows represent equivalences and those labeled with $d$ represent dual equivalences.

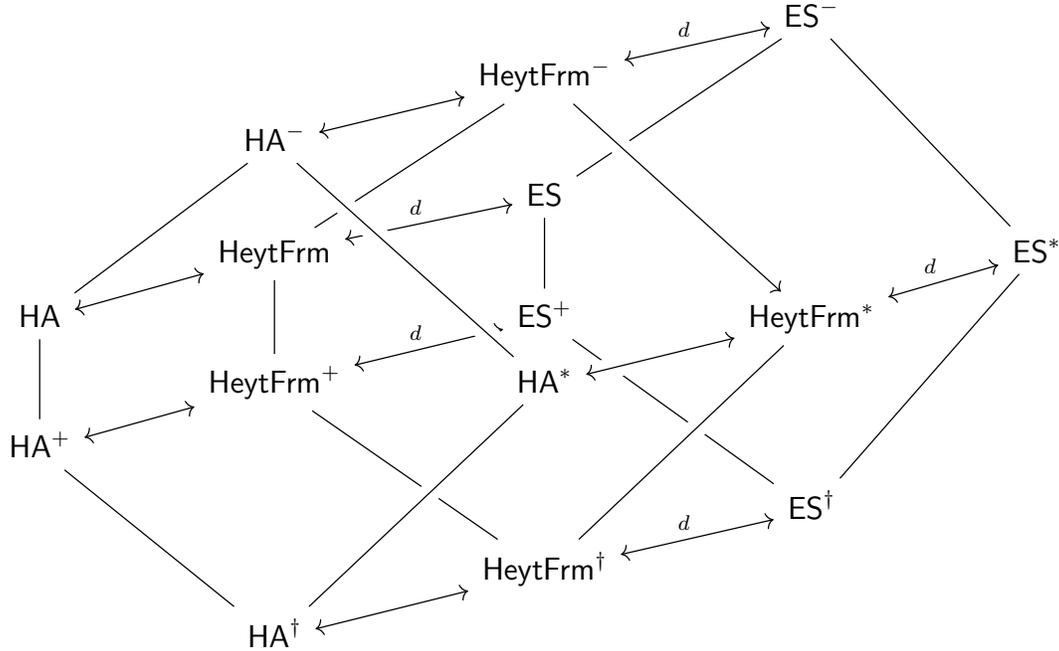
\[
\begin{tikzcd}[column sep = 1.75pc, row sep = .3pc]
&&&&&& \Esa^- \arrow[rrdddd, dash] && \\
&&&& \HeytFrm^- \arrow[ddddrr, dash] \arrow[urr, leftrightarrow, "d"] &&&& \\
&& \HA^-  \arrow[ddddrr, dash] \arrow[urr, leftrightarrow] &&&&&& \\
&&&& \Esa  \arrow[uuurr, dash] &&&& \\
&& \HeytFrm \arrow[urr, leftrightarrow, "d"] \arrow[uuurr, dash]  &&&&&& \Esa^* \arrow[lldddd, dash] \\
\HA  \arrow[urr, leftrightarrow] \arrow[uuurr, dash] &&&& \Esa^+ \arrow[uu, dash] \arrow[dddrr, dash] && \HeytFrm^* \arrow[urr, leftrightarrow, "d"] \arrow[ddddll, dash, crossing over] \arrow[from=lluuuu, crossing over] && \\
&& \HeytFrm^+ \arrow[dddrr, dash] \arrow[urr, leftrightarrow, "d"] \arrow[uu, dash] && \HA^* \arrow[urr, leftrightarrow, crossing over] \arrow[lldddd, dash, crossing over]  \arrow[from=uuuull, dash, crossing over] &&&& \\
\HA^+ \arrow[urr, leftrightarrow] \arrow[uu, dash] &&&&&&&& \\
&&&&&& \Esa^\dagger && \\
&&&& \HeytFrm^\dagger  \arrow[urr, leftrightarrow, "d"]  &&&& \\
&& \HA^\dagger  \arrow[urr, leftrightarrow] \arrow[lluuu, dash] &&&&&&
\end{tikzcd}
\]

\section{Brouwerian frames} \label{sec: Brouwerian frames}

In this final section we show how to generalize the results of the previous section to Brouwerian semilattices and Brouwerian algebras. We recall that a {\em meet-semilattice} is a poset $A$ in which all finite meets exist. In particular, $A$ has a top, but $A$ may not have a bottom. A meet-semilattice is {\em distributive} if $a\wedge b\le c$ implies that there exist $a'\ge a$ and $b'\ge b$ such that $a'\wedge b'=c$. A \emph{meet-semilattice homomorphism} is a map between meet-semilattices that preserves finite meets (including top). Let $\DMSLat$ be the category of distributive meet-semilattices and meet-semilattice homomorphisms.

\begin{definition}
Let $A$ be a meet-semilattice. 
\begin{enumerate}
\item We call $A$ a \emph{Brouwerian semilattice} or an \emph{implicative semilattice} if it has an implication operation $\to$ satisfying $x \le a \to b$ iff $a \wedge x \le b$. Let $\BrwMS$ be the category of Brouwerian semilattices and maps which preserve finite meets and implication.
\item We call a Brouwerian semilattice $A$ a {\em Brouwerian algebra} if in addition $A$ is a lattice. Let $\BrwA$ be the category of Brouwerian algebras and Brouwerian semilattice homomorphisms that in addition preserve $\vee$.
\end{enumerate}
\end{definition}

Note that $\BrwA$ is a non-full subcategory of $\BrwMS$. Also, each Brouwerian semilattice is a distributive meet-semilattice. Thus, $\BrwMS$ is a non-full subcategory of $\DMSLat$. 

We first generalize Esakia duality to Brouwerian algebras. For this we need to work with pointed Esakia spaces.

\begin{definition}
A {\em pointed Esakia space} is a pair $(X,m)$ where $X$ is an Esakia space and $m$ is the unique maximum of $X$. Let $\PEsa$ be the category of pointed Esakia spaces and Esakia morphisms.
\end{definition}

We note that if $(X,m)$ and $(Y,n)$ are pointed Esakia spaces and $f:X\to Y$ is an Esakia morphism, then $f(m)=n$. Esakia duality generalizes to Brouwerian algebras as follows:

\begin{theorem} \cite[Thm.~3.2]{BMR17} \label{thm: duality for BrwA}
$\BrwA$ is dually equivalent to $\PEsa$.
\end{theorem}

\begin{remark}
The contravariant functors establishing Esakia duality need slight modification. Namely, with a pointed Esakia space $(X,m)$ we associate the Brouwerian algebra of {\em nonempty} clopen upsets of $X$ (equivalently, those clopen upsets of $X$ that contain $m$). Also, with a Brouwerian algebra $A$ we associate the pointed Esakia space $(X,m)$ where $X$ is the set of prime filters of $A$ together with $A$, and $A$ serves as the unique maximum $m$ of $X$. The action of the functors on morphisms is the same as in Esakia duality.
\end{remark}

We next generalize Esakia duality to Brouwerian semilattices. A {\em pointed Priestley space} is a pair $(X,m)$ where $X$ is a Priestley space and $m$ is the unique maximum of $X$. Let $X_0$ be a fixed subset of $X\setminus\{m\}$. We call a subset $U$ of $X$ \emph{admissible} if $X_0\setminus U$ is cofinal in $X\setminus U$ (that is, $X \setminus U \subseteq \down (X_0 \setminus U)$). Let $\A(X)$ be the set of admissible clopen upsets of $X$. For $x \in X$ set $\I_x = \{ U \in \A(X) : x \notin U\}$. 

\begin{definition}\label{def: PGEsa}
\hfill
\begin{enumerate}
\item A {\em pointed generalized Priestley space} is a triple $(X,X_0,m)$ such that
\begin{enumerate}
\item $(X,m)$ is a pointed Priestley space;
\item $X_0$ is a cofinal dense subset of $X \setminus \{ m \}$;
\item $x \in X_0$ iff $\I_x$ is nonempty and directed;
\item $x \le y$ iff $(\forall U \in \A(X))(x \in U \Rightarrow y \in U)$.
\end{enumerate}
\item A \emph{pointed generalized Esakia space} is a pointed generalized Priestley space in which $U, V \in \A(X)$ implies that $\down (U \setminus V)$ is clopen.
\end{enumerate}
\end{definition}

\begin{remark}
Let $(X,X_0,m)$ be a pointed generalized Esakia space. In analogy with \cite[Def.~3.4]{BJ13}, we call a subset $E$ of $X$
\emph{Esakia clopen}, or simply \emph{E-clopen}, if $E=\bigcup_{i=1}^n (U_i \setminus V_i)$ with each $U_i, V_i \in \A(X)$. It is easy to see from \cref{def: PGEsa}(2) that 
$\down E$ is clopen for each E-clopen $E$. However, not every clopen of $X$ is E-clopen. Thus, a pointed generalized Esakia space may not be an Esakia space (see \cite[Exmp.~3.9]{BJ13}).
\end{remark}

Let $R \subseteq X \times Y$ be a relation between sets $X$ and $Y$. For $U \subseteq Y$ we follow the standard notation in modal logic and write $\Box_RU = \{ x\in X : R[x] \subseteq U\}$.

\begin{definition} \label{def: morphism}
\hfill
\begin{enumerate}
\item Let $X, Y$ be pointed generalized Priestley spaces. A {\em generalized Priestley morphism} is a relation $R \subseteq X \times Y$ satisfying
\begin{enumerate}
\item If $x \nr{R} y$, then there is $U \in \A(Y)$ with $R[x] \subseteq U$ and $y \notin U$;
\item If $U \in \A(Y)$, then $\Box_RU \in \A(X)$.
\end{enumerate}
\item Let $\PGPries$ be the category of pointed generalized Priestley spaces and generalized Priestley morphisms. 
\item If $X, Y$ are pointed generalized Esakia spaces, then a generalized Priestley morphism $R \subseteq X \times Y$ is a \emph{generalized Esakia morphism} provided whenever $x \in X$ and $y \in Y_0$ with $xRy$, there is $z \in X_0$ with $x \le z$ and $R[z] = \up y$.
\item Let $\PGEsa$ be the category of pointed generalized Esakia spaces and generalized Esakia morphisms.
\end{enumerate}
\end{definition}

\begin{remark}
Composition in $\PGPries$ and $\PGEsa$ is not usual relation compostion. Rather, for generalized Priestley (resp.~Esakia) morphisms $R \subseteq X \times Y$ and $S \subseteq Y \times Z$, the composition $S \ast R \subseteq X \times Z$ is defined by 
\[
x (S \ast R) z \Longleftrightarrow (\forall U \in \A(Z))(x\in\Box_R\Box_S U \Longrightarrow z\in U) 
\]
for all $x\in X$ and $z\in Z$ (see \cite[p.~106]{BJ11}).
\end{remark}

Clearly $\PGEsa$ is a non-full subcategory of $\PGPries$, and Priestley and Esakia dualities are generalized to distributive and Brouwerian semilattices as follows: 

\begin{theorem} \label{thm: duality for BrwMS}
\hfill
\begin{enumerate}[ref = \thetheorem(\arabic*), label = $(\arabic*)$] 
\item \cite[Thm.~6.9]{BJ11} $\DMSLat$ is dually equivalent to $\PGPries$. \label[theorem]{thm: duality for BrwMS(1)}
\item \cite[Thm.~4.4]{BJ13} $\BrwMS$ is dually equivalent to $\PGEsa$. \label[theorem]{thm: duality for BrwMS(2)}
\end{enumerate}
\end{theorem}

\begin{remark}
\cref{thm: duality for BrwMS} is proved in \cite{BJ11,BJ13} for distributive and Brouwerian semilattices that are bounded. Because of this restriction, there is no need to work with pointed spaces. The need for pointed spaces arises when the bottom is not present. This is discussed in detail in \cite{BCM22a} for distributive meet-semilattices, and a similar approach also works for Brouwerian semilattices.
\end{remark}

\begin{remark} \label{rem: A and X}
The contravariant functor $\A : \PGPries \to \DMSLat$ sends $X \in \GPries$ to $\A(X)$ and a $\PGPries$-morphism $R \subseteq X \times Y$ to $\Box_R : \A(Y) \to \A(X)$. To define the contravariant functor $\X : \DMSLat \to \PGPries$, let $A\in\DMSLat$. We recall that a filter $F$ of $A$ is \emph{optimal} if from $a_1, \dots, a_n \notin F$ and $\bigcap_{i=1}^n \up a_i \subseteq \up c$ it follows that $c \notin F$. Equivalently, if $D$ is the distributive envelope of $A$ (see, e.g., \cite[Sec.~3]{BJ11}), then $F$ is optimal iff $F = P \cap A$ for some prime filter $P$ of $D$ (\cite[Prop.~4.8]{BJ11}).  Let $\Opt(A)$ be the set of optimal filters of $A$ and ${\sf Pr}(A)$ the set of prime filters of $A$. Then ${\sf Pr}(A) \subseteq \Opt(A)$. We set $X_A = \Opt(A) \cup \{A\}$, order it by inclusion, and topologize it by letting 
\[
\{ \varphi(a) : a \in A\} \cup \{ \varphi(b)^c : b \in A\}
\] 
be a subbasis for the topology, where $\varphi(a) = \{ x \in X_A : a \in x \}$.\footnote{Since $\varphi(a\wedge b)=\varphi(a)\cap\varphi(b)$ for each $a,b\in A$, the subbasis $\{ \varphi(a) : a \in A\} \cup \{ \varphi(b)^c : b \in A\}$ generates the basis 
$
\{ \varphi(a) \cap \varphi(b_1)^c \cap \cdots \cap \varphi(b_n)^c : a, b_1, \dots, b_n \in A\}.
$ 
However, unlike in the case of distributive lattices, we cannot replace the finite intersections $\varphi(b_1)^c \cap \cdots \cap \varphi(b_n)^c$ with one $\varphi(b)^c$ because $A$ is not a lattice.} Then $\X(A) := (X_A, {\sf Pr}(A), A)$ is a pointed generalized Priestley space. If $h : A \to B$ is a $\DMSLat$-morphism, then $\X(h) = R_h$, where $x R_h y$ if $h^{-1}(x) \subseteq y$. The functors $\A$ and $\X$ yield the dual equivalence between $\PGPries$ and $\DMSLat$ of \cref{thm: duality for BrwMS(1)}, which further restricts to the dual equivalence of \cref{thm: duality for BrwMS(2)}.
\end{remark}

We next connect distributive and implicative semilattices with algebraic frames. Let $A$ be a distributive meet-semilattice. Since $A$ is not a lattice, instead of working with ideals of $A$, it is more convenient to work with filters of $A$. Let $\F(A)$ be the poset of filters of $A$ ordered by inclusion. Then $\F(A)$ is an algebraic frame whose compact elements are the principal filters of $A$. But since $a \le b$ iff $\up a \supseteq \up b$, we have that $A$ is isomorphic to the order-dual $K(\F(A))^d$ of $K(\F(A))$.

If $h : A \to B$ is a meet-semilattice homomorphism, we can define $\F(h) : \F(A) \to \F(B)$ by $\F(h)(F) = \up h[F]$. Note that in general $\F(h)$ preserves arbitrary joins, but may not be a frame homomorphism. In order for $\F(h)$ to be a frame homomorphism, we need two additional conditions on $h$. We recall that a \emph{prime filter} of a meet-semilattice $A$ is a meet-prime element of $\mathcal{F}(A)$.

\begin{definition} \cite[Def.~5.15]{BCM22a} \label{def: sup-homs}
Let $\DMSLatP$ be the category of distributive meet-semilattices and meet-semilattice homomorphisms $h : A \to B$ that in addition satisfy:
\begin{enumerate}
\item $h$ preserves all existing finite joins.
\item $h^{-1}(P)$ is a prime filter of $A$ for each prime filter $P$ of $B$.
\end{enumerate}
\end{definition}

Meet-semilattice homomorphisms satisfying the above two conditions correspond to special functions between the corresponding generalized Priestley spaces:

\begin{definition} \label{def: morphismP}
Let $\PGPriesP$ be the category of pointed generalized Priestley spaces and order-preserving maps $f : X \to Y$ satisfying 
\begin{enumerate}
\item $U \in \A(Y) \Longrightarrow f^{-1}(U) \in \A(X)$.
\item $f[X_0] \subseteq Y_0$. 
\end{enumerate}
\end{definition}

\begin{theorem} \cite[Thm.~5.18]{BCM22a} \label{thm: Pont}
$\DMSLatP$ is equivalent to $\AlgFrm$, dually equivalent to $\PGPriesP$, and the following diagram commutes up to natural isomorphism. 
\[
\begin{tikzcd}
\DMSLatP \arrow[rr, shift left = .5ex, "\F"] \arrow[dr,  shift left=.5ex, "\X"] && \AlgFrm \arrow[ll, shift left = .5ex, "\K"] \arrow[dl, shift left = .5ex, "\Y"] \\
& \PGPriesP \arrow[ul, shift left = .5ex, "\A"] \arrow[ur, shift left = .5ex, "\V"] &
\end{tikzcd}
\]
\end{theorem}

\begin{remark} \label{rem: description of functors}
We briefly describe the functors in the above diagram. 
\begin{itemize}[leftmargin=*]
\item The contravariant functors $\X$ and $\A$ are described in \cref{rem: A and X}.
\item The covariant functor $\F:\DMSLatP\to\AlgFrm$ sends $A \in \DMSLatP$ to the frame $\F(A)$ of filters of $A$ and a $\DMSLatP$-morphism $h : A \to B$ to $\F(h) : \F(A) \to \F(B)$ given by $\F(h)(F) = \up h[F]$. The covariant functor $\K$ sends $L \in \AlgFrm$ to $K(L)^d$ and an $\AlgFrm$-morphism $\alpha : L \to M$ to its restriction to $K(L)^d$, where $K(L)^d$ is the order-dual of $K(L)$. The functors $\F$ and $\K$ yield an equivalence of $\DMSLatP$ and $\AlgFrm$.

\item To describe the contravariant functor $\V:\PGPriesP \to \AlgFrm$, for $X \in \PGPries$ let $\V(X)$ be the set of admissible closed upsets of $X$, ordered by reverse inclusion. Then $\V$ sends $X\in\PGPriesP$ to $\V(X)$ and a $\PGPriesP$-morphism $f : X \to Y$ to $f^{-1} : \V(Y)\to\V(X)$. To describe the contravariant functor $\Y : \AlgFrm \to \PGPriesP$, we recall the notions of pseudoprime and prime elements. If $L$ is a frame and $1 \ne p \in L$, then $p$ is (\emph{meet}-)\emph{prime} if $a, b \in L$ and $a \wedge b \le p$ imply that $a \le p$ or $b \le p$. Moreover, $p$ is \emph{pseudoprime} if $a_1, \dots, a_n \in L$ with $a_1 \wedge \cdots \wedge a_n \ll p$ imply that $a_i \le p$ for some $i$, where $\ll$ is the way below relation on $L$ (see, e.g., \cite[p.~49]{GHKLMS03}). Let $PP(L)$ and $P(L)$ be the sets of pseudoprime and prime elements of $L$, respectively. We set $Y_L = PP(L) \cup \{1\}$, order it by the restriction of the order on $L$, and topologize it by the subbasis
\[
\{\up k \cap Y_L : k \in K(L)\} \cup \{(\up l)^c \cap Y_L : l \in K(L)\}. 
\]
Then the functor $\Y$ sends $L \in \AlgFrm$ to $\Y(L) := (Y_L, P(L), 1)$. 
To describe the action of $\Y$ on morphisms, let $\alpha : L \to M$ be an $\AlgFrm$-morphism. Then it has the right adjoint $r : M \to L$ given by 
\[
r(b)=\bigvee\{ a\in L : \alpha(a)\le b \},
\]
and $\Y$ sends $\alpha$ to $r$. The functors $\V$ and $\Y$ yield a dual equivalence of $\PGPriesP$ and $\AlgFrm$.
\end{itemize}
\end{remark}

\begin{remark}
It follows from \cref{thm: Pont} that $\X\circ \K$ is naturally isomorphic to $\Y$. This implies that prime elements of $L$ correspond to prime filters of $K(L)^d$ and pseudoprime elements of $L$ to optimal filters of $K(L)^d$. It is well known (see, e.g., \cite[pp.~13-14]{PP12}) that prime elements of $L$ correspond to points of $L$.
On the other hand, pseudoprime elements correspond to what we term ``pseudopoints'' of $L$, which are defined as follows. A nonempty upset $U$ of $L$ is a \emph{pseudopoint} if 
\begin{itemize}
\item $\bigvee S \in U$ implies that $S \cap U \ne \varnothing$;
\item $a_1, \dots, a_n \in U$ and $a_1 \wedge \cdots \wedge a_n \ll b$ imply that $b \in U$.
\end{itemize}
\end{remark}

\begin{definition}
\hfill
\begin{enumerate}
\item Let $\BrwMSP$ be the category of Brouwerian semilattices and $\BrwMS$-morphisms that are also $\DMSLatP$-morphisms.
\item Let $\BrwAP$ be the full subcategory $\BrwMSP$ consisting of Brouwerian algebras.
\end{enumerate}
\end{definition}

We are ready to introduce the key notion of this section, that of a Brouwerian frame. These frames play the same role for Brouwerian semilattices as Heyting frames for Heyting algebras. 

\begin{definition} 
\hfill
\begin{enumerate}
\item We call an algebraic frame $L$ a {\em Brouwerian frame} if $K(L)^d$ 
is a Brouwerian semilattice. 
\item Let $\BrwFrm$ be the category whose objects are Brouwerian frames and whose morphisms are $\AlgFrm$-morphisms $\alpha : L \to M$ such that $\alpha(a \to b) = \alpha(a) \to \alpha(b)$ for each $a, b \in K(L)$, where the two implications are calculated in $K(L)^d$ and $K(M)^d$, respectively.
\end{enumerate}
\end{definition}

\begin{remark}
We don't have an analogue of \cref{def: HeytFrm} because due to turning the order around, implication becomes co-implication, which may not exist in frames.
\end{remark}

\begin{definition}
Let $\PGEsaP$ be the category of pointed generalized Esakia spaces and $\PGPriesP$-morphisms $f : X \to Y$ satisfying: If $x \in X$, $y \in Y_0$, and $f(x) \le y$, then there is $z \in X_0$ with $x \le z$ and $y = f(z)$.
\end{definition}

We show that the functors of \cref{thm: Pont} restrict to yield an equivalence between $\BrwMSP$ and $\BrwFrm$ and a dual equivalence between $\BrwFrm$ and $\PGEsaP$. For this we require the following lemma, which follows from \cite{BJ11,BJ13}. 
Slight care is needed since these papers only consider the bounded case, while we do not assume the existence of a bottom. Nonetheless, the relevant proofs carry over to our more general setting.

\begin{lemma} \label{lem: PGEsa facts}
\hfill
\begin{enumerate}[ref= \thelemma(\arabic*), label = $(\arabic*)$] 
\item Let $X$ be a pointed generalized Priestley space. Then $X$ is a pointed generalized Esakia space iff $\A(X)$ is a Brouwerian semilattice.
\label[lemma]{lem: PGEsa facts(1)}
\item Let $X,Y$ be pointed generalized Esakia spaces and $f:X\to Y$ a $\PGPriesP$-morphism. Then $f$ is a $\PGEsaP$-morphism iff $f^{-1}(U \to V) = f^{-1}(U) \to f^{-1}(V)$ for each $U, V \in \A(Y)$. \label[lemma]{lem: PGEsa facts(2)}
\end{enumerate}
\end{lemma}

\begin{proof}

(1) follows from \cite[Prop.~3.7]{BJ13} and \cite[Thm.~5.13]{BJ11}  
and (2) follows from 
\cite[Props.~4.1, 4.3, Lem.~4.8]{BJ13}.
\end{proof}

\begin{theorem} \label{thm: BrwMS = BrwFrm}
There is an equivalence of categories between $\BrwMSP$ and $\BrwFrm$, a dual equivalence between $\BrwFrm$ and $\PGEsaP$, and the following diagram commutes up to natural isomorphism, where the functors are the restrictions of the corresponding functors of Theorem~\emph{\ref{thm: Pont}}.
\[
\begin{tikzcd}
\BrwMSP \arrow[rr, shift left = .5ex, "\F"] \arrow[dr,  shift left=.5ex, "\X"] && \BrwFrm \arrow[ll, shift left = .5ex, "\K"] \arrow[dl, shift left = .5ex, "\Y"] \\
& \PGEsaP \arrow[ul, shift left = .5ex, "\A"] \arrow[ur, shift left = .5ex, "\V"] &
\end{tikzcd}
\]
\end{theorem}

\begin{proof}
We  show that $\F,\K$ restrict to give an equivalence between $\BrwMSP$ and $\BrwFrm$. If $A \in \DMSLatP$, then since $A \cong K(\F(A))^d$, we see that $\F(A) \in \BrwFrm$ iff $A \in \BrwMSP$. Next, let $h : A \to B$ be a $\DMSLatP$-morphism between Brouwerian semilattices. Then $h$ is a $\BrwMSP$-morphism iff it preserves implication, which happens iff $\F(h)$ is a $\BrwFrm$-morphism, again by the isomorphisms $A \cong K(\F(A))^d$ and $B \cong K(\F(B))^d$. Finally, $\DMSLatP$-isomorphisms between Brouwerian semilattices are clearly $\BrwMSP$-isomorphisms, and similarly $\AlgFrm$-isomorphisms between Brouwerian frames are $\BrwFrm$-isomorphisms. Thus, $\F$ and $\K$ restrict to an equivalence between $\BrwMSP$ and $\BrwFrm$ by \cref{rem: category theory}.

We next show that $\V,\Y$ restrict to give a dual equivalence between $\PGEsaP$ and $\BrwFrm$.  Let $X \in \PGPriesP$. By \cref{lem: PGEsa facts(1)}, $X \in \PGEsaP$ iff $\A(X)\in\BrwMSP$. By \cite[Lem.~4.4]{BCM22a}, $\V(X)$ is an algebraic frame with $K(\V(X))^d = \A(X)$. Hence, $\A(X)\in\BrwMSP$ iff $\V(X) \in \BrwFrm$. Thus, $X \in \PGEsaP$ iff $\V(X) \in \BrwFrm$. Let $f : X \to Y$ be a $\PGPriesP$-morphism between pointed generalized Esakia spaces. Since $K(\V(Y))^d = \A(Y)$, \cref{lem: PGEsa facts(2)} shows that $f$ is a $\PGEsaP$-morphism iff $f^{-1}(U \to V) = f^{-1}(U) \to f^{-1}(V)$ for each $U, V \in K(\V(Y))$. Therefore, $f$ is a $\PGEsaP$-morphism iff 
$f^{-1}$ is a $\BrwFrm$-morphism. It is also clear that $\PGPriesP$-isomorphisms between pointed generalized Esakia spaces are $\PGEsaP$-isomorphisms. Thus, $\V$ and $\Y$ restrict to an equivalence between $\PGEsaP$ and $\BrwFrm$.

In view of \cref{thm: Pont}, it follows that the diagram commutes up to natural isomorphism.
\end{proof}

\begin{remark} \label{rem: BrwFrmJ}
If in \cref{thm: BrwMS = BrwFrm} we replace $\BrwMSP$ with $\BrwMS$ and $\PGEsaP$ with $\PGEsa$, then we need to weaken the notion of a $\BrwFrm$-morphism $\alpha : L \to M$ by dropping the condition that $\alpha$ preserves finite meets. If we denote the resulting category of Brouwerian frames by $\BrwFrmJ$ (where J stands for join-preserving), then we obtain the following version of the diagram in \cref{thm: BrwMS = BrwFrm} which commutes up to natural isomorphism:
\[
\begin{tikzcd}
\BrwMS \arrow[rr, shift left = .5ex, "\F"] \arrow[dr,  shift left=.5ex, "\X"] && \BrwFrmJ \arrow[ll, shift left = .5ex, "\K"] \arrow[dl, shift left = .5ex, "\Y"] \\
& \PGEsa \arrow[ul, shift left = .5ex, "\A"] \arrow[ur, shift left = .5ex, "\V"] &
\end{tikzcd}
\]
The proof is an appropriate modification of the proof of \cref{thm: BrwMS = BrwFrm} along the lines of \cite[Thm.~4.30]{BCM22a}.
\end{remark}

\begin{remark}
The other categories of Heyting frames considered in \cref{sec: Heyting frames} also have natural generalizations to the setting of Brouwerian frames. The equivalences and dual equivalences of \cref{sec: Heyting frames} then generalize to involve the corresponding categories of Brouwerian frames. 
\end{remark}

We next study those Brouwerian frames that correspond to Brouwerian algebras. Recall (see, e.g., \cite[p.~117]{GHKLMS03}) that an algebraic frame $L$ is {\em arithmetic} if $a,b\in K(L)$ implies $a\wedge b\in K(L)$. 

\begin{definition}
Let $\BrwArFrm$ be the full subcategory of $\BrwFrm$ whose objects are 
Brouwerian arithmetic frames.
\end{definition}

\begin{definition}
Let $\PEsaP$ be the wide subcategory of $\PEsa$ whose morphisms $f : (X, m) \to (Y, n)$ satisfy $f^{-1}(\{n\}) = \{m\}$.
\end{definition}

\begin{remark}
For a pointed Esakia space $(X,m)$, if we let $X_0=X\setminus\{m\}$, then $(X,X_0,m)$ is a pointed generalized Esakia space. Thus, we can view $\PEsaP$ as a full subcategory of $\PGEsaP$.
\end{remark}

\begin{theorem}\label{thm: BrwArFrm}
The functors of Theorem~\emph{\ref{thm: BrwMS = BrwFrm}} restrict to yield that $\BrwAP$ is equivalent to $\BrwArFrm$ and dually equivalent to $\PEsaP$. Consequently, the following diagram commutes up to natural isomorphism.
\[
\begin{tikzcd}
\BrwAP \arrow[rr, shift left = .5ex, "\F"] \arrow[dr,  shift left=.5ex, "\X"] && \BrwArFrm \arrow[ll, shift left = .5ex, "\K"] \arrow[dl, shift left = .5ex, "\Y"] \\
& \PEsaP \arrow[ul, shift left = .5ex, "\A"] \arrow[ur, shift left = .5ex, "\V"] &
\end{tikzcd}
\]
\end{theorem}

\begin{proof}
Let $L \in \BrwArFrm$. We prove that $\Y(L) \in \PEsaP$. For this it suffices to show that $PP(L) \subseteq P(L)$. Let $p \in PP(L)$ and $a,b \in L$ with $a,b \not\le p$. Since $L$ is an algebraic frame, there are $k, l \in K(L)$ with $k \le a$, $l \le b$, and $k, l \not\le p$. Because $L$ is arithmetic, $k \wedge l \in K(L)$. If $a \wedge b \le p$, then $k \wedge l \ll p$. Since $p \in PP(L)$, either $k \le p$ or $l \le p$. The obtained contradiction shows that $a \wedge b \not \le p$, and hence $p \in P(L)$.

Next let $X \in \PEsaP$. Then each nonempty closed upset is admissible, and so if $U, V \in \A(X)$, then $U \cup V \in \A(X)$ since $\A(X)  = K(\V(X))^d$. This shows that $\V(X)$ is arithmetic, and hence $\V(X) \in \BrwArFrm$. 

It is also elementary to see that if $L \in \BrwFrm$, then $L \in \BrwArFrm$ iff $\K(L) \in \BrwAP$. 
Thus, since $\BrwAP$, $\BrwArFrm$, and $\PEsaP$ are full subcategories of $\BrwMSP$, $\BrwFrm$, and $\PGEsaP$, the result follows from \cref{thm: BrwMS = BrwFrm}.
\end{proof}

\begin{remark} 
As in \cref{rem: BrwFrmJ}, if we replace $\BrwAP$ with $\BrwA$ and $\PEsaP$ with $\PEsa$, then we have to replace $\BrwArFrm$ with $\BrwArFrmJ$ to obtain the following diagram of equivalences and dual equivalences that commutes up to natural isomorphism:
\[
\begin{tikzcd}
\BrwA \arrow[rr, shift left = .5ex, "\F"] \arrow[dr,  shift left=.5ex, "\X"] && \BrwArFrmJ \arrow[ll, shift left = .5ex, "\K"] \arrow[dl, shift left = .5ex, "\Y"] \\
& \PEsa \arrow[ul, shift left = .5ex, "\A"] \arrow[ur, shift left = .5ex, "\V"] &
\end{tikzcd}
\]
\end{remark}

Combining \cref{thm: BrwMS = BrwFrm,thm: BrwArFrm,thm: HA case}, we arrive at the following diagram which commutes up to natural isomorphism:
\[
\begin{tikzcd}
\BrwMSP \arrow[r, leftrightarrow] & \BrwFrm \arrow[r, leftrightarrow, "d"] &\PGEsaP \\
\BrwAP \arrow[u] \arrow[r, leftrightarrow] & \BrwArFrm \arrow[u]  \arrow[r, leftrightarrow, "d"] &\PEsaP \arrow[u] \\
\HA \arrow[u] \arrow[r, leftrightarrow] & \HeytFrm \arrow[u] \arrow[r, leftrightarrow, "d"] & \Esa \arrow[u]
\end{tikzcd}
\]
The horizontal arrows represent equivalences or dual equivalences when the label $d$ is present. The vertical arrows represent full subcategories, except $\HeytFrm$ is not really a subcategory of $\BrwArFrm$ (see below). 
The equivalences and dual equivalences of the middle row are restrictions of the equivalences and dual equivalences of the top row. The situation with the bottom row is slightly different in that if $L\in\HeytFrm$, then we work with $K(L)$ rather than $K(L)^d$. Similarly, if $X\in\Esa$, we work with the open upsets of $X$ rather than the closed upsets (ordered by reverse inclusion). 

If instead of $\HeytFrm$ we worked with the category whose objects are Brouwerian arithmetic frames in which $K(L)^d$ is a Heyting algebra and whose morphisms are $\BrwArFrm$-morphisms, 
then we would obtain a category that is a full subcategory of $\BrwArFrm$ and is equivalent to $\HeytFrm$. 
In other words, if instead of working with the frames of open upsets of Esakia spaces we worked with the frames of closed upsets  (ordered by reverse inclusion), then we would obtain a category that is equivalent to a full subcategory of $\BrwArFrm$ that is equivalent to $\HeytFrm$.

\end{document}